\documentclass[reqno,12pt]{amsart}
\newcommand{\RR}{{\mathbb R}}
\newcommand{\CC}{{\mathbb C}}

\def\bege{\begin{equation}} \def\ende{\end{equation}}
\def\begr{\begin{eqnarray}} \def\endr{\end{eqnarray}}
\usepackage{amsmath}
\usepackage{amssymb}
\usepackage{cite}
\usepackage{bbm}
\usepackage{amsmath}
\usepackage{txfonts}
\usepackage{stmaryrd}
\usepackage{amssymb}
\usepackage{amsfonts}
\usepackage{times}
\usepackage{mathrsfs}
\usepackage{amsthm}
\usepackage{enumerate}
\usepackage{framed}
\usepackage{lipsum}
\usepackage{color}

\def\CC{ \mathbb{C}}
\newcommand{\DD}{{\mathbb D}}

\def\begr{\begin{eqnarray}} \def\endr{\end{eqnarray}}

\def\msk{\medskip}

\newtheorem{Lemma}{Lemma}
\newtheorem{Theorem}{Theorem}
\newtheorem{Corollary}{Corollary}

\begin{document}
	\title[  ]{  $C$-normal weighted composition operators on $H^2$ }
	\author{ Lian Hu,  Songxiao Li$^\ast$ and  Rong Yang }
	%%%
	\address{Lian Hu\\ Institute of Fundamental and Frontier Sciences, University of Electronic Science and Technology of China,
		610054, Chengdu, Sichuan, P.R. China.}
	\email{hl152808@163.com  }
\address{Songxiao Li\\ Institute of Fundamental and Frontier Sciences, University of Electronic Science and Technology of China,
		610054, Chengdu, Sichuan, P.R. China. }
	 \email{jyulsx@163.com}

	\address{Rong Yang\\ Institute of Fundamental and Frontier Sciences, University of Electronic Science and Technology of China,
		610054, Chengdu, Sichuan, P.R. China.}
	\email{yangrong071428@163.com }

	\begin{abstract}
	A bounded linear operator $T$ on a separable complex Hilbert space $H$ is called $C$-normal if there is a conjugation $C$ on $H$ such that $ CT^\ast TC=TT^\ast$.	Let $\varphi$ be a linear fractional self-map of $\DD$.  In this paper, we characterize the necessary and sufficient condition for the composition operator $C_\varphi$   and weighted composition operator $W_{\psi,\varphi}$  to be $C$-normal with some conjugations $C$ and a function $\psi$.
	\thanks{$\ast$ Corresponding author.}
\end{abstract}
\maketitle

\section{Introduction}
In this paper, $\DD$ always denotes the open unit disk and $\partial \DD$ denotes the unit circle in the complex plane $\mathbb{C}$. Let $H(\DD)$ be the space of  all analytic functions on $\DD$. The space $H^\infty(\DD)$ consists of all bounded  analytic functions on $\DD$.  The Hardy space $H^2(\DD)$ consists of all analytic functions $f(z)=\sum_{n=0}^{\infty} a_{n} z^{n}$ on $\mathbb{D}$ such that
$$
\|f\|^{2}=\sum_{n=0}^{\infty}\left|a_{n}\right|^{2}<\infty.
$$   For any $w\in \DD$, the function $K_w={1\over 1-\bar{w}z}$ is the reproducing kernel at $ w$ in $H^2(\DD)$ such that $$
\langle f,K_w\rangle=f(w)$$ for any $f\in H^2(\DD)$.

 Let $\varphi$ be an analytic self-map of $\DD$. A composition operator $C_{\varphi}$ on $H^{2}(\DD)$ is given by $$C_{\varphi} f=f \circ \varphi.$$
  The operator $C_{\varphi}$ is automatically bounded on $H^2(\DD)$. Let $\varphi$ be an analytic self-map of $\DD$ and $\psi \in H(\DD)$. The weighted composition operator $W_{\psi,\varphi}$  is defined by
$$W_{\psi,\varphi} f(z)=\psi(z) f(\varphi(z)), ~~~~~~~ z\in \DD,~~~~~ f\in H(\DD).$$
Assuming that $W_{\psi,\varphi}$ is bounded on $H^2(\DD)$, then  $$W_{\psi,\varphi}^\ast K_w(z)=\overline{\psi(w)}K_{\varphi(w)}(z)$$ for any $w,z\in \DD$.
 In recent decades, the study of composition operators and weighted composition operators has been widely concerned and developed rapidly. See \cite{cm, s} and reference therein for general background of them.

For $g\in L^\infty(\partial\DD)$, the Toeplitz operator $T_g$ is given by $T_gf=P(fg)$ for $f\in H^2(\DD)$, where $P$ is the orthogonal projection of $L^2$ onto $H^2(\DD)$. $T_{\psi} f=\psi f$ when $\psi \in H^{\infty}$. Therefore, we occasionally write $W_{\psi, \varphi}=T_{\psi} C_{\varphi}$ when $\psi \in H^{\infty}$. In addition, $$T_g^\ast K_w=\overline{g(w)}K_w$$ for any $w\in \DD$ and $g\in H^\infty(\DD)$.

%An operator $T\in \mathcal{B}(\mathcal{H})$ is called $p$-hyponormal if $ (TT^\ast)^p-(T^\ast T)^p\le0$, where $0<p\le1$. In particular, if $p=1$, then $T$ is hyponormal. We say that $T\in \mathcal{B}(\mathcal{H})$ is quasinormal if $T$ commutes with $T^\ast T$, that is $[T,T^\ast T]=0$, where $[A,B]=AB-BA$.

 A conjugate-linear (anti-linear) operator $C:\mathcal{H}\to \mathcal{H}$ is called a  conjugation  if it satisfies $$\langle Cx, Cy \rangle_\mathcal{H}=\langle y, x \rangle_\mathcal{H}$$ for any $x,y \in \mathcal{H}$ and $C^2=I_\mathcal{H}$, where $I_\mathcal{H}$ is an identity operator on $\mathcal{H}$.
It is easy to check that $(Jf)(z)=\overline{f(\bar{z})}$ is a conjugation on the space $H^2(\DD)$.

Let $\mathcal{H}$ and $\mathcal{B}(\mathcal{H})$ be a separable complex Hilbert space and the space of all bounded linear operators on $\mathcal{H}$, respectively.  An operator $T\in \mathcal{B}(\mathcal{H})$ is said to be normal if $TT^\ast=T^\ast T$. Let $C$ be a conjugation on $\mathcal{H}$ and $T \in \mathcal{B}(\mathcal{H})$.   $T $ is said to be $C$-symmetric if $T^\ast=C T C$.  %It is called $C$-skew-symmetric if $C T C=-T^\ast$.
$T $ is said to be complex symmetric if there exists a conjugation $C$ such that $T$ is $C$-symmetric. Any normal composition operator induced by $\varphi(z)=\alpha z$  with $|\alpha|\leq 1$  is complex symmetric.  In fact, the class of complex symmetric operators contain  those normal operators, Hankel operators, binormal operators, compressed Toeplitz operators and the Volterra integration operator. For applications and further information about complex symmetric operators,  see \cite{f,gz,gp,gp2,gw,gw2,lz, nst, z2,z3}.

   In \cite{psw}, Ptak, Simik and Wicher introduced a new interesting class of operators, named $C$-normal operators.   $T\in \mathcal{B}(\mathcal{H})$ is called $C$-normal if $$ CT^\ast T C=TT^\ast.$$  Obviously,   $T$ is  $C$-normal if and only if $ T^\ast$ is $C$-normal. Recently, Wang, Zhao and Zhu studied the structure of $C$-normal operators in \cite{wzz}. For more results of $C$-normal operators, see \cite{kll,psw,rrn,wzz}.

In this paper,  we investigate the $ J_\mu$-normal   and $ JW_{\xi_p, \tau_p}$-normal   composition operator $C_\varphi$ and weighted composition operator $W_{\psi,\varphi}$  when the induced mapping  $\varphi $ is a linear fractional self-map of $\DD$.   Organization of this article as follows: In Section 2, we characterize some necessary and sufficient conditions for  $C_\varphi$ to be   $ J_\mu$-normal    and $ JW_{\xi_p, \tau_p}$-normal. In  Section 3,  we study some properties of the symbols $\varphi$ and $\psi$ whenever the  operator $W_{\psi, \varphi}$ is $ J_\mu$-normal    and $ JW_{\xi_p, \tau_p}$-normal  on  $H^2(\DD)$.  Moreover, we also find the conditions that $ J_\mu$-normal  or $ JW_{\xi_p, \tau_p}$-normal weighted composition operators are  unitary, Hermitian and normal operators.  \msk

  \section{$C$-normal composition operators}
  In this section, we study $C$-normal composition operators on the Hardy space $H^2(\DD)$ when the conjugation $C$ is   defined by
   $$
  \mathcal{A}_{u,v}f(z)=u(z)\overline{f(\overline{v(z)})},\,\,f\in H^2(\DD),$$
  where $u:\DD\to \CC$ and $v:\DD\to \DD$ are analytic functions.  It is clear that $\mathcal{A}_{u,v}=J$   when $u(z)\equiv1$ and $v(z)=z.$

\begin{Lemma}\label{lem7}\cite[Lemma 2.11]{lk}
	Let $u:\DD\to \CC$ and $v:\DD\to \DD$ are analytic functions. Then $\mathcal{A}_{u,v}:H^2(\DD)\to H^2(\DD) $ is a conjugation if and only if either of the following holds.
		\begin{enumerate}
		\item[(i)] There is $\beta,\mu\in \partial\DD$ such that for any $z\in \DD$, $$
		u(z)=\beta{ ~and~}v(z)=\mu z.$$
		\item[(ii)] There is  $p\in \DD\setminus \{0\}$   and $\beta\in \partial\DD$ such that for any $z\in \DD$, $$
		 u(z)=\beta{\sqrt{1-|p|^2} \over1-pz }{ ~and~}v(z)={p \over \bar{p}}{ \bar{p}-z \over 1-pz}.$$
		\end{enumerate}
\end{Lemma}

  \begin{Lemma}\cite[Lemma 2]{jhz}\label{lem1}
  	Suppose that $$\xi_p(z)={\sqrt{1-|p|^2}  \over 1-\bar{p}z } \,\,\,and \,\,\, \tau_p(z)={\lambda(p-z) \over 1-\bar{p}z}\in {\rm{Aut}(\DD)},
   $$
  	where $p\in\DD$ and $|\lambda|=1$. Then $JW_{\xi_p, \tau_p}$ is a conjugation on $H^2(\DD)$ if and only if $\lambda p=\bar{p}$.
  \end{Lemma}

As a direct consequence of Lemma \ref{lem7}, there are only two classes of conjugations $ \mathcal{A}_{u,v}$ on $H^2(\DD)$. In the first class, choosing $u(z)=1$ and $v(z)=\mu z$, we obtain that  $$ J_\mu f(z)=\mathcal{A}_{u,v}f(z)= \overline{f(\mu\bar{z})} $$
 is a conjugation on $H^2(\DD)$. In particular, $J=J_1$.  In the second class, we choose
$$
u(z)={\sqrt{1-|p|^2} \over1-pz }{ ~and~}v(z)={p \over \bar{p}}{ \bar{p}-z \over 1-pz}.$$
Applying Lemmas \ref{lem7} and \ref{lem1}, we see that $  \mathcal{A}_{u,v}=JW_{\xi_p, \tau_p}$
 is a conjugation on $H^2(\DD)$, where $p\in \DD\setminus \{0\}$.

 Next we will characterize $ J_\mu$-normal composition operators  and $ JW_{\xi_p, \tau_p}$-normal composition operators which are induced by linear fractional self-map of $\DD$.
   For this purpose, we need some lemmas.

    \begin{Lemma}\label{lem3}\cite{cm}
  	Let $\varphi$ be an analytic self-map of  $\DD$. Then $C_\varphi$ is normal if and only if $\varphi(z)=\alpha z$ with $|\alpha|\le 1$.
  \end{Lemma}

   Let $\varphi(z)={az+b\over cz+d }$ be a nonconstant linear fractional self-map of $\DD$,  where $a,b,c,d\in \mathbb{C}$ and $ad-bc\ne 0$.
  Cowen in \cite{c2} proved that
  $$C_\varphi^\ast =T_gC_\sigma T_h^\ast,$$
  where
  $$
  \sigma(z)={\bar{a}z-\bar{c}  \over -\bar{b}z+\bar{d} },~~g(z)={1\over-\bar{b}z+\bar{d} },~~ h(z)=cz+d.$$

\begin{Theorem}\label{the3}
	Let $\varphi(z)={az+b\over cz+d }$ be a linear fractional self-map of $\DD$. Then $C_\varphi$ is $J_\mu$-normal if and only if $C_\varphi$ is normal.
\end{Theorem}

\begin{proof}
	Assume first that $C_\varphi$ is $J_\mu $-normal. Let
	$$
	\sigma(z)={\bar{a}z-\bar{c}  \over -\bar{b}z+\bar{d} },~~g(z)={1\over-\bar{b}z+\bar{d} },~~ h(z)=cz+d.$$
	For any $w,z\in \DD$, we have
	\begin{equation}\label{equ5}
		\begin{aligned}
			C_\varphi C_\varphi^\ast J_\mu K_w(z)=&C_\varphi C_\varphi^\ast\overline {K_w(\mu\bar{z})}=C_\varphi C_\varphi^\ast{1\over 1-\bar{\mu}wz}\\
			=& 	C_\varphi C_\varphi^\ast K_{\mu \bar{w}}(z)=K_{\varphi{(\mu\bar{w})}}(\varphi(z))\\
			=&{1\over  1-{ \overline{a\mu}w+\bar{b} \over \overline{c\mu}w+\bar{d} }\cdot{az+b\over cz+d}}\\
			=&{ (\overline{c\mu}w+\bar{d})(cz+d) \over (\overline{c\mu}w+\bar{d})(cz+d)-(\overline{a\mu}w+\bar{b})(az+b) }.
		\end{aligned}
	\end{equation}	
	For $\bar{a}w\ne\bar{c}$, Lemma 2.1 in \cite{jyk} yields that
	\begin{equation}\label{equ6}
		\begin{aligned}	
			&J_\mu C_\varphi^\ast C_\varphi K_w(z)\\
				=&J_\mu\left( -c{ \overline{g(w)} \over \overline{ \sigma(w)}}K_{\varphi(0)}(z)+h \left({1  \over \overline{\sigma(w)}}\right) \overline{g(w)}K_{\varphi( { \sigma(w) })}(z) \right)\\
			=& -\bar{c}{ g(w) \over \sigma(w) }K_{\mu\overline{\varphi(0)}}({z})+\overline{h( {1  \over \overline{\sigma(w)} } )}g(w)K_{\mu\overline{\varphi( { \sigma(w) })}}(z)\\
			=&	{\bar{c}\over \bar{c}-\bar{a}w}\cdot{1\over  1-{b\over d}\mu z}	+\left( {\bar{d}\over  \bar{d}-\bar{b}w}- {\bar{c}\over \bar{c}-\bar{a}w}\right)\cdot{1\over 1- \mu z\cdot{ (|a|^2-|b|^2)w+b\bar{d}-a\bar{c} \over (\bar{a}c-\bar{b}d)w+|d|^2-|c|^2 } }.
		\end{aligned}
	\end{equation}	
 	Since $C_\varphi$ is $J_\mu $-normal, then
	\begin{align}\label{equ7}
		C_\varphi C_\varphi^\ast J_\mu K_w(z)=J_\mu C_\varphi^\ast C_\varphi K_w(z)
	\end{align}
	for any $w,z\in \DD$. Taking $w=0$, we have
	\begin{align*}
		{ \bar{d}(cz+d) \over \bar{d}(cz+d)-\bar{b}(az+b) }={d\over d-b\mu z}
	\end{align*}
	for any $z\in \DD$. Noting that the coefficient of constant in the last equality must be zero, which means that $|b|^2d=0$. Thus, $b=0$. Then $$\left| d\over a \right|\ge \left| c\over a \right| +1.$$ (\ref{equ5}), (\ref{equ6}) and (\ref{equ7}) deduce that
	\begin{align*}
		&{ (\overline{c\mu}w+\bar{d})(cz+d) \over (\overline{c\mu}w+\bar{d})(cz+d)-|a|^2\bar{\mu}wz }\\
		=&{\bar{c}\over \bar{c}-\bar{a}w}	+\left( 1- {\bar{c}\over \bar{c}-\bar{a}w}\right)\cdot{1\over 1-\bar{\mu}z\cdot{ |a|^2w-a\bar{c} \over \bar{a}cw+|d|^2-|c|^2 } }	\\
		=&{\bar{c}\over \bar{c}-\bar{a}w}	-{ \bar{a}w\over \bar{c}-\bar{a}w}\cdot{ \bar{a}cw+|d|^2-|c|^2  \over \bar{a}cw+|d|^2-|c|^2 -(|a|^2w-a\bar{c} ) \bar{\mu}z },
	\end{align*}	
	for any $w,z\in \DD$, which implies that
	\begin{equation}\label{equ8}
		\begin{aligned}	
			&{ (\overline{c\mu}w+\bar{d})(cz+d)(\bar{c}-\bar{a}w)-\bar{c}[(-|a|^2\bar{\mu}w+|c|^2\bar{\mu}w+\bar{d}c)z+|d|^2+\bar{c}d\bar{\mu}w] \over [(-|a|^2\bar{\mu}w+|c|^2\bar{\mu}w+\bar{d}c)z+|d|^2+\bar{c}d\bar{\mu}w](\bar{c}-\bar{a}w) }\\
			=&{ -\bar{a}w(\bar{a}cw+|d|^2-|c|^2) \over \bar{a}cw+|d|^2-|c|^2-(|a|^2w-a\bar{c})\bar{\mu}z }
		\end{aligned}
	\end{equation}	
	for any $w,z\in \DD$. Noting that the coefficient of $z^2$ in (\ref{equ8}) must be zero, that is
	\begin{align}\label{equ9}
		(-|c|^2\overline{a\mu}w^2-\bar{a}c\bar{d}w+\bar{c}|a|^2\bar{\mu}w)(|a|^2w-ac)=0
	\end{align}
	for any $w\in \DD$. Similarly, the coefficient of $w^3$ in (\ref{equ9}) must be zero, that is,  $ \bar{a}|c|^2|a|^2\bar{\mu}^2=0$. Since $|\mu|=1$ and $a\ne 0$, we get  $c=0$. Therefore, $\varphi(z)=\alpha z$ with $|\alpha|\le 1$ and Lemma \ref{lem3} gives that $C_\varphi$ is normal.
	
	Conversely, suppose that $C_\varphi$ is normal. Lemma \ref{lem3} yields that $\varphi(z)=\alpha z$ with $|\alpha|\le 1$. Then$$
	C_\varphi C_\varphi^\ast J_\mu K_w(z)={ 1\over 1-|a|^2\bar{\mu}wz}=J_\mu C_\varphi^\ast C_\varphi K_w(z)$$
	for any $w,z\in \DD$. So $C_\varphi$ is $J_\mu$-normal.  The proof is complete.
\end{proof}

As a consequence of Theorem  \ref{the3}, we have the following result.

\begin{Corollary}
	Let $\varphi(z)={az+b\over cz+d }$ be a linear fractional self-map of $\DD$. Then the following statements are equivalent.
	\begin{enumerate}
		\item[(i)] $C_\varphi$ is normal;
		\item[(ii)] $C_\varphi$ is $J$-normal;
		\item[(iii)] $C_\varphi$ is $J_\mu$-normal.
	\end{enumerate}	
\end{Corollary}

   \begin{Lemma}\label{lem2}
  	Let $p\in \DD\setminus \{0\}$, $\varphi(z)={ az+b \over cz+d }$ be a linear fractional self-map of $\DD$. If $w\in\DD$ satisfies $\bar{a}w\ne\bar{c}$, then the following statements hold:
  	\begin{enumerate}
  		\item[(i)]  $C_\varphi C_\varphi^\ast JW_{\xi_p, \tau_p} K_w(z)=\overline{\xi_p(\bar{w})}K_{\varphi{(\eta)}}(\varphi(z))$,
  		\item[(ii)]
  		$$JW_{\xi_p, \tau_p}C_\varphi^\ast C_\varphi K_w(z)=\overline{\xi_p(\bar{z})}\left( -\bar{c}{ g(w) \over \sigma(w) }K_{\overline{\varphi(0)}}(\overline{\tau_p{(\bar{z})}})+\overline{h( {1  \over \overline{\sigma(w)} } )}g(w)K_{\overline{\varphi( { \sigma(w) })}}(\overline{\tau_p{(\bar{z})}})\right),$$
  		  	\end{enumerate}
  	where $JW_{\xi_p, \tau_p} $  is a conjugation defined in Lemma \ref{lem1}, and $$\eta={\bar{p}-\bar{w}\lambda  \over 1-\overline{wp} },~ \sigma(z)={ \bar{a}z-\bar{c} \over -\bar{b}z+\bar{d} },~ g(z)={1\over -\bar{b}z+\bar{d}},  h(z)=cz+d.$$
  \end{Lemma}

  \begin{proof}
  First, we verify $(i)$. For any $w,z\in \DD$, we have
  \begin{align*}
  	C_\varphi C_\varphi^\ast JW_{\xi_p, \tau_p} K_w(z)=&	C_\varphi C_\varphi^\ast J {\xi_p({z})}K_{{w}}({\tau_p{({z})}})\\
  	=&C_\varphi C_\varphi^\ast J{\sqrt{1-|p|^2}  \over 1-\bar{p}z }\cdot {1\over {1- { \bar{w}\lambda(p-z) \over 1-\bar{p}z} }}\\
  	=&C_\varphi C_\varphi^\ast J{ \sqrt{1-|p|^2}  \over 1-\overline{wp}+(\bar{w}\lambda-\bar{p})z }\\
  	=&C_\varphi C_\varphi^\ast{  \sqrt{1-|p|^2} \over 1-wp+(w\bar{\lambda}-p)z}\\
  	=&{  \sqrt{1-|p|^2} \over 1-wp }C_\varphi C_\varphi^\ast K_\eta(z)=\overline{\xi_p(\bar{w})}K_{\varphi{(\eta)}}(\varphi(z)).
  \end{align*}
%  where $\eta={\bar{p}-\bar{w}\lambda  \over 1-\overline{wp} }$.

  	 Now, we verify $(ii)$. Applying Lemma 2.1 in \cite{jyk}, we have
  	 \begin{align*}
  	 	&JW_{\xi_p, \tau_p}C_\varphi^\ast C_\varphi K_w(z)\\
  	 	=&JW_{\xi_p, \tau_p}\left( -c{ \overline{g(w)} \over \overline{\sigma(w)} }K_{\varphi(0)}(z)+h( {1  \over \overline{\sigma(w)} } )\overline{g(w)}K_{\varphi( { \sigma(w) })}(z) \right)\\
  	 	=&\overline{\xi_p(\bar{z})}\left( -\bar{c}{ g(w) \over \sigma(w) }K_{\overline{\varphi(0)}}(\overline{\tau_p{(\bar{z})}})+\overline{h( {1  \over \overline{\sigma(w)} } )}g(w)K_{\overline{\varphi( { \sigma(w) })}}(\overline{\tau_p{(\bar{z})}})\right)	.
  	 \end{align*} The proof is complete.
  \end{proof}

Next, we study $JW_{\xi_p, \tau_p}$-normal composition operators which are induced by linear fractional self-map of $\DD$ with $\varphi(0)=0$.

\begin{Theorem}
	Let $p\in \DD\setminus \{0\}$ and  $\varphi(z)={az+b\over cz+d }$ be a linear fractional self-map of $\DD$ such that $\varphi(0)=0$.   Then $C_\varphi$ is $JW_{\xi_p, \tau_p} $-normal if and only if $C_\varphi$ is isometry.
\end{Theorem}

\begin{proof}
	Suppose first that $C_\varphi$ is $JW_{\xi_p, \tau_p} $-normal. Since $\varphi(0)=0 $, let $$
	\sigma(z)={ \bar{a}z-\bar{c} \over \bar{d} },~ g(z)={1\over \bar{d}},~and~ h(z)=cz+d.$$
   For $\bar{a}w\ne\bar{c}$, using Lemma \ref{lem2}, we obtain
		\begin{align*}
		&C_\varphi C_\varphi^\ast JW_{\xi_p, \tau_p} K_w(z)= \overline{\xi_p(\bar{w})}K_{\varphi{(\eta)}}(\varphi(z))\\
		=&{  \sqrt{1-|p|^2} \over 1-wp }\cdot{  1\over  1-{ -\overline{a\lambda}w+\bar{a}p\over (-\overline{c\lambda}-\bar{d}p)w+\bar{c}p+\bar{d} }\cdot{az  \over cz+d }}\\
		=&{  \sqrt{1-|p|^2} \over 1-wp }\cdot{ (cz+d)[(-\overline{c\lambda}-\bar{d}p)w+\bar{c}p+\bar{d}] \over  (cz+d)[(-\overline{c\lambda}-\bar{d}p)w+\bar{c}p+\bar{d}]-az(-\overline{a\lambda}w+\bar{a}p)},
		\end{align*}		
where $\eta={\bar{p}-\bar{w}\lambda  \over 1-\overline{wp} }$, 	and
	\begin{align*}
	&JW_{\xi_p, \tau_p}C_\varphi^\ast C_\varphi K_w(z)\\
	=&\overline{\xi_p(\bar{z})}\left( -\bar{c}{ g(w) \over \sigma(w) }K_{\overline{\varphi(0)}}(\overline{\tau_p{(\bar{z})}})+\overline{h( {1  \over \overline{\sigma(w)} } )}g(w)K_{\overline{\varphi( { \sigma(w) })}}(\overline{\tau_p{(\bar{z})}})\right)\\
	=&{  \sqrt{1-|p|^2} \over 1-pz }\left( { \bar{c} \over \bar{c}-\bar{a}w }+(1-{ \bar{c} \over \bar{c}-\bar{a}w }) \cdot {1\over  1-{ |a|^2w-a\bar{c} \over \bar{a}cw+|d|^2-|c|^2 }\cdot{ p-\bar{\lambda}z \over 1-pz }} \right)\\
	= & { \bar{c}\sqrt{1-|p|^2} \over (\bar{c}-\bar{a}w )(1-pz)}-{ \bar{a}w \over ( \bar{c}-\bar{a}w ) }\cdot{\sqrt{1-|p|^2}(\bar{a}cw+|d|^2-|c|^2)  \over (\bar{a}cw+|d|^2-|c|^2)(1-pz)-(|a|^2w-a\bar{c})(p-\bar{\lambda}z) }.
\end{align*}
Since $C_\varphi$ is $JW_{\xi_p, \tau_p} $-normal, then $$
C_\varphi C_\varphi^\ast JW_{\xi_p, \tau_p} K_w(z)=JW_{\xi_p, \tau_p}C_\varphi^\ast C_\varphi K_w(z)$$
for any $w,z\in \DD$. Taking $w=0$, we get
\begin{align*}
	{ \sqrt{1-|p|^2}(\bar{c}p+\bar{d})(cz+d) \over (\bar{c}p+\bar{d})(cz+d)-|a|^2pz }
	={ \sqrt{1-|p|^2} \over 1-pz }
	\end{align*}
for any $z\in \DD$. This means that
	\begin{align}\label{equ3}
 { (|c|^2p+c\bar{d})z+\bar{c}dp+|d|^2 \over [(|c|^2-|a|^2)p+c\bar{d}]z +\bar{c}dp+|d|^2}={1\over 1-pz}
 \end{align}
 for any $z\in \DD$. Therefore, the coefficients of $z$ and $z^2$ in (\ref{equ3}) must be zero, which deduce that
 	\begin{align}\label{equ4}
 |c|^2p^2+c\bar{d}p=0~{\rm and} ~ |a|^2p=p(\bar{c}dp+|d|^2).
\end{align}
 If $c\ne0$, since $p\in \DD\backslash\{0\}$ and $d\ne 0$, (\ref{equ4}) gives that $ \bar{c}p+\bar{d}=0$. Thus, $$|a|^2=d(\bar{c}p+\bar{d})=0,$$ which means that $a=0$. In this case, $\varphi$ is not an analytic self-map. So, $c=0$ and (\ref{equ4}) yields that $ |a|^2=|d|^2$, that is $|a|=|d|$. Hence, $\varphi(z)=\alpha z$, where $|\alpha|=1$. Then $C_\varphi $ is isometry (see \cite{n}).

 Conversely,
 assume that $C_\varphi $ is isometry. By a simple calculation,  we have that $\varphi(z)=\alpha z$, where $|\alpha|=1$. Then
  \begin{align*}
  C_\varphi C_\varphi^\ast JW_{\xi_p, \tau_p} K_w(z)=&{  \sqrt{1-|p|^2} \over 1-wp }\cdot{1\over 1-{ \overline{\alpha \lambda}w+\bar{\alpha}p \over 1-pw }\cdot\alpha z}\\
  =&{  \sqrt{1-|p|^2} \over 1-pw-pz-\bar{\lambda}wz }\\
  =&JW_{\xi_p, \tau_p}C_\varphi^\ast C_\varphi K_w(z).
  \end{align*}
 Hence, $C_\varphi$ is $JW_{\xi_p, \tau_p} $-normal. The proof is complete.
	\end{proof}

\begin{Theorem}
	Let $p\in \DD\setminus \{0\}$ and $\varphi(z)={az+b\over cz+d }$ be a linear fractional self-map of $\DD$ such that $\varphi(0)\ne 0$. Then $C_\varphi$ is not $JW_{\xi_p, \tau_p} $-normal.
\end{Theorem}

\begin{proof}
	Let $$
	\sigma(z)={\bar{a}z-\bar{c}  \over -\bar{b}z+\bar{d} },~~g(z)={1\over-\bar{b}z+\bar{d} },~~ h(z)=cz+d.$$
For $\bar{a}w\ne\bar{c}$, employing Lemma \ref{lem2}, we obtain
	\begin{equation}\label{equ10}
		\begin{aligned}
			&C_\varphi C_\varphi^\ast JW_{\xi_p, \tau_p} K_w(z)=\overline{\xi_p(\bar{w})}K_{\varphi{(\eta)}}(\varphi(z))\\
			=&{  \sqrt{1-|p|^2} \over 1-wp }\cdot{  1\over  1-{ (-\overline{a\lambda}-\bar{b}p)w+\bar{a}p+\bar{b}\over (-\overline{c\lambda}-\bar{d}p)w+\bar{c}p+\bar{d} }\cdot{az +b \over cz+d }}\\
			=&{  \sqrt{1-|p|^2} \over 1-wp }\cdot{ (cz+d)[(-\overline{c\lambda}-\bar{d}p)w+\bar{c}p+\bar{d}] \over  (cz+d)[(-\overline{c\lambda}-\bar{d}p)w+\bar{c}p+\bar{d}]-(az+b)[(-\overline{a\lambda}-\bar{b}p)w+\bar{a}p+\bar{b}]},
		\end{aligned}
	\end{equation}
where $\eta={\bar{p}-\bar{w}\lambda  \over 1-\overline{wp} }$,	and
	\begin{equation}\label{equ11}
		\begin{aligned}
			&JW_{\xi_p, \tau_p}C_\varphi^\ast C_\varphi K_w(z)\\
			=&\overline{\xi_p(\bar{z})}\left( -\bar{c}{ g(w) \over \sigma(w) }K_{\overline{\varphi(0)}}(\overline{\tau_p{(\bar{z})}})+\overline{h( {1  \over \overline{\sigma(w)} } )}g(w)K_{\overline{\varphi( { \sigma(w) })}}(\overline{\tau_p{(\bar{z})}})\right)\\
			=&{  \sqrt{1-|p|^2} \over 1-pz }\left( { \bar{c} \over \bar{c}-\bar{a}w }\cdot{1\over 1-{b\over d}\cdot{ p-\bar{\lambda}z \over 1-pz}}+\left({ \bar{d} \over \bar{d}-\bar{b}w}-{ \bar{c} \over \bar{c}-\bar{a}w }\right) \cdot {1\over  1-{ (|a|^2-|b|^2)w+b\bar{d}-a\bar{c} \over (\bar{a}c-\bar{b}d)w+|d|^2-|c|^2 }\cdot{ p-\bar{\lambda}z \over 1-pz }} \right).
			%= & { \bar{c}\sqrt{1-|p|^2} \over (\bar{c}-\bar{a}w )(1-pz)}-{ \bar{a}w \over ( \bar{c}-\bar{a}w ) }\cdot{\sqrt{1-|p|^2}(\bar{a}cw+|d|^2-|c|^2)  \over (\bar{a}cw+|d|^2-|c|^2-(|a|^2w-a\bar{c})(p-\bar{\lambda}z)) }.
		\end{aligned}
	\end{equation}
Suppose that $C_\varphi$ is $JW_{\xi_p, \tau_p} $-normal. Then for any $w,z\in \DD$
	\begin{align}\label{equ12}
 C_\varphi C_\varphi^\ast JW_{\xi_p, \tau_p} K_w(z)=JW_{\xi_p, \tau_p}C_\varphi^\ast C_\varphi K_w(z).
 \end{align}
 Taking $w=0$ in (\ref{equ12}),
(\ref{equ10}), (\ref{equ11}) and (\ref{equ12}) give that
\begin{equation}\label{equ19}
	\begin{aligned}
		{(\bar{c}p+\bar{d})(cz+d)  \over (\bar{c}p+\bar{d})(cz+d)-(\bar{a}p+\bar{d})(az+b) }=&{1\over1-pz}\cdot{1\over 1-{b\over d}\cdot{p-\bar{\lambda} z\over  1-pz}}\\
		=&{d\over d(1-pz)-b(p-\bar{\lambda} z)}
	\end{aligned}
\end{equation}
for any $z\in \DD$.

Next, we divide the rest proof into    three cases.

Case 1. $c=0$, $b\ne 0$. (\ref{equ19}) gives that
 %	\begin{equation}\label{equ13}
 %	\begin{aligned}
 %	&	{  \sqrt{1-|p|^2} \over 1-pw } \cdot{ d(-\bar{d}pw+\bar{d}) \over d(-\bar{d}pw+\bar{d})- (az+b)[(-\overline{a\lambda}-\bar{b}p)w+\bar{a}p+\bar{b}]}\\
 %	=&{  \bar{d}\sqrt{1-|p|^2} \over (1-pz )(\bar{d}-\bar{b}w)}\cdot{ 1\over 1-{(|a|^2-|b|^2)w+b\bar{d} \over -\bar{b}dw+|d|^2}\cdot{ p-\bar{\lambda}z \over 1-pz }}\\
% 	=&{  \bar{d}\sqrt{1-|p|^2}(-\bar{b}dw+|d|^2) \over (\bar{d}-\bar{b}w)[(-\bar{b}dw+|d|^2)(1-pz)-[(|a|^2-|b|^2)w+b\bar{d}](p-\bar{\lambda}z)]}
 %	\end{aligned}
% \end{equation}
%for any $w,z\in \DD$. Taking $w=0$ in (\ref{equ13}), we have
\begin{align}\label{equ14}
	|d|^2-(az+b)(\bar{a}p+\bar{b})=|d|^2(1-pz)-b\bar{d}(p-\bar{\lambda}z)
\end{align}	
for any $z\in \DD$. Therefore, the coefficient of constant in (\ref{equ14}) must be zero, that is $ \bar{d}p=\bar{b}+\bar{a}p$. Then (\ref{equ10}), (\ref{equ11}) and (\ref{equ12}) give that
\begin{equation}\label{equ15}
	\begin{aligned}
   &{  \sqrt{1-|p|^2} \over 1-wp }\cdot{d(-\bar{d}pw+\bar{d})  \over d(-\bar{d}pw+\bar{d}) -(az+b)[(-\overline{a\lambda}-\bar{b}p)w+\bar{a}p+\bar{b}] }\\
   =&{\sqrt{1-|p|^2} \over 1-pz}\cdot {\bar{d}\over \bar{d}-\bar{b}w}\cdot{  1\over  1-{ (|a|^2-|b|^2)w+b\bar{d} \over -\bar{b}dw+|d|^2 }\cdot{ p-\bar{\lambda}z \over  1-pz}}\\
 		\end{aligned}
 \end{equation}
for any $w,z\in \DD$, which is equivalent to
\begin{equation}\label{equ25}
	\begin{aligned}
&{|d|^2\over |d|^2-b\bar{d}p-[\bar{b}d+p(|a|^2-|b|^2)]w+[\bar{\lambda}(|a|^2-|b|^2)w+\bar{\lambda}b\bar{d}-|d|^2p]z}	\\
=&	{|d|^2\over |d|^2-|b|^2-\bar{a}bp+[\bar{a}b\bar{\lambda}+(|b|^2-|d|^2)p]w+[(|a|^2\bar{\lambda}+a\bar{b}p)w+|a|^2p+a\bar{b}]z}
\end{aligned}
\end{equation}		
for any $w,z\in \DD$. Therefore, by comparing the coefficient of $zw$ in (\ref{equ25}),  we obtain that $$
|a|^2\bar{\lambda}+a\bar{b}p=\bar{\lambda}(|a|^2-|b|^2),$$which means that $ \bar{a}p+\bar{b}=0$. This contradicts with $ \bar{d}p=\bar{b}+\bar{a}p$.

  Case 2. $a=0$, $b\ne 0$, $c\ne 0$. (\ref{equ19}) gives that
%  \begin{equation}\label{equ17}
%  	\begin{aligned}
%  		&{  \sqrt{1-|p|^2} \over 1-wp }\cdot{ (cz+d)[(-\overline{c\lambda}-\bar{d}p)w+\bar{c}p+\bar{d}] \over  (cz+d)[(-\overline{c\lambda}-\bar{d}p)w+\bar{c}p+\bar{d}]+|b|^2pw-|b|^2}\\
%  		=&{  \sqrt{1-|p|^2} \over 1-pz }\left( {1\over 1-{b\over d}\cdot{ p-\bar{\lambda}z \over 1-pz}}+\left({ \bar{d} \over \bar{d}-\bar{b}w}-1\right) \cdot {1\over  1-{ -|b|^2w+b\bar{d} \over -\bar{b}dw+|d|^2-|c|^2 }\cdot{ p-\bar{\lambda}z \over 1-pz }} \right)
%\end{aligned}
%\end{equation}
%for any $w,z\in \DD$. Taking $w=0$ in (\ref{equ17}), we obtain
\begin{align}\label{equ18}
	{ (cz+d)(\bar{c}p+\bar{d}) \over  (cz+d)(\bar{c}p+\bar{d})-b\bar{d}}={d\over d(1-pz)-b(p-\bar{\lambda}z)}
	\end{align}
for any $z\in \DD$. By comparing the coefficients of $z$, $z^2$ and constant in (\ref{equ18}), we have $$
bpc(\bar{c}p+\bar{d})=d(\bar{c}p+\bar{d})(b\bar{\lambda}-dp),$$ $$(\bar{c}p+\bar{d})(b\bar{\lambda}-dp)=0, ~{\rm and}~bp(\bar{c}p+\bar{d})=b\bar{d},$$
which implies that $b\bar{d}=0$. This is a contradiction.

  Case 3. $a\ne 0$, $b\ne 0$, $c\ne 0$. By comparing the coefficients of $z$, constant and $z^2$  in (\ref{equ19}), we obtain
  $$(\bar{c}p+\bar{d})(b\bar{\lambda}-dp)=0,
  p(\bar{c}p+\bar{d})=(\bar{a}p+\bar{b})$$and$$ad(\bar{a}p+\bar{b})=cbp(\bar{c}p+\bar{d})-d(\bar{c}p+\bar{d})(b\bar{\lambda}-dp).
  $$
Since  $ ad-bc\ne0$, we get
  $$
  ad(\bar{a}p+\bar{b})=cbp(\bar{c}p+\bar{d})=cb(\bar{a}p+\bar{b}),$$
  which implies that $ (\bar{a}p+\bar{b})=(\bar{c}p+\bar{d})=0$. This contradicts with $ ad-bc\ne0$. The proof is complete.
\end{proof}

   \section{ $C$-normal weighted composition operators}

  In this section, we characterize $J_\mu$-normal and $JW_{\xi_p, \tau_p}$-normal  weighted composition operators $W_{\psi,\varphi}$ which are induced by linear fractional self-map $\varphi$ of $\DD$ and specific function $\psi$, which given by  $\psi(z)=\beta K_{\sigma(0)}(z)$.

\begin{Theorem}\label{the5}
	Let $\varphi(z)={az+b\over cz+d }$ be a linear fractional self-map of $\DD$ and $\psi(z)=\beta K_{\sigma(0)}(z)$, where $\beta $  is a non-zero constant and $\sigma(z)={ \bar{a}z-\bar{c} \over -\bar{b}z+\bar{d} }$. Then $W_{\psi,\varphi}$ is $J_\mu $-normal if and only if$$|b|=|c| ~and~	(\bar{c}d-\bar{a}b)\bar{\mu}=\bar{a}c-\bar{b}d.$$
\end{Theorem}

\begin{proof}
	For any $w,z\in \DD$, we have
	\begin{align*}
		&	J\mu W_{\psi,\varphi}W_{\psi,\varphi}^\ast k_w(z)=J\mu W_{\psi,\varphi}\overline{\beta K_{\sigma(0)}(w)}K_{\varphi(w)}(z)\\
		=&	{\beta}J\mu \beta K_{\sigma(0)}(z)K_{\overline{\sigma(0)}}(w)K_{\varphi(w)}(\varphi(z))\\
		=&|\beta|^2\overline{K_{\sigma(0)}(\mu \bar{z})}K_{\sigma(0)}(w)\overline{K_{\varphi(w)}(\varphi(\mu \bar{z}))}\\
		=&|\beta|^2K_{\sigma(0)}(w)K_{\mu \overline{\sigma(0)}}(z)K_{\overline{\varphi(w)}}(\overline{\varphi(\mu \bar{z})})\\
		=&{ |\beta|^2 \over 1+{c\over d}w }\cdot{1  \over 1+{\bar{c}\over \bar{d}}\bar{\mu}z }\cdot{1\over  1-{ aw+b \over cw+d }{ \overline{a\mu}z+\bar{b} \over \overline{c\mu}z+\bar{d} }}\\
		=&{|\beta|^2|d|^2  \over (d+cw)(\bar{d}+\overline{c\mu}z)-(aw+b)(\overline{a\mu}z+\bar{b}) }\\
		=&{ |\beta|^2|d|^2  \over [(|c|^2-|a|^2)\bar{\mu}w+(\bar{c}d-\bar{a}b)\bar{\mu}]z+(c\bar{d}-a\bar{b})w+|d|^2-|b|^2 }.
		\end{align*}
	Let $g(z)={1\over -\bar{b}z+\bar{d}}$ and $h(z)=cz+d$. Noting that
 $$
	W_{\psi,\varphi}^\ast=T_gC_\sigma T_h^\ast T_\psi^\ast {\rm ~and~}\beta K_{\sigma(0)}h=\beta d,$$
for any $w,z\in \DD$, we obtain that
	\begin{align*}
		&W_{\psi,\varphi}^\ast W_{\psi,\varphi}J_\mu K_w(z)=W_{\psi,\varphi}^\ast W_{\psi,\varphi}\overline{K_w(\mu\bar{z})}\\=&W_{\psi,\varphi}^\ast W_{\psi,\varphi}K_{\bar{w}\mu}(z)
		=T_gC_\sigma T_h^\ast T_\psi^\ast \psi(z)K_{\bar{w}\mu}(\varphi(z))\\
		=&T_gC_\sigma\left( T_{\beta K_{\sigma(0)}} T_h\right)^\ast\beta K_{\sigma(0)}(z)K_{\bar{w}\mu}(\varphi(z))\\
		=&T_gC_\sigma|\beta|^2\bar{d}K_{\sigma(0)}(z)K_{\bar{w}\mu}(\varphi(z))\\
		=&|\beta|^2\bar{d}g(z)K_{\sigma(0)}(\sigma(z))K_{\bar{w}\mu}(\varphi(\sigma(z)))\\
		=&{ |\beta|^2\bar{d} \over -\bar{b}z+\bar{d} }\cdot{1\over  1+{c\over d}{ \bar{a}z-\bar{c} \over -\bar{b}z+\bar{d} }}\cdot{  1\over 1-w\bar{\mu}{(|a|^2-|b|^2)z+b\bar{d}-a\bar{c}  \over (\bar{a}c-\bar{b}d)z+|d|^2-|c|^2 }}\\
		=&{ |\beta|^2\bar{d} \over (\bar{a}c-\bar{b}d)z+|d|^2-|c|^2 -w\bar{\mu}[(|a|^2-|b|^2)z+b\bar{d}-a\bar{c}] }\\
		=&{ |\beta|^2\bar{d} \over [(|a|^2-|b|^2)\bar{\mu}w+(\bar{a}c-\bar{b}d)]z+(a\bar{c}-b\bar{d})\bar{\mu}w+|d|^2-|c|^2 }.
		\end{align*}
	Therefore, $W_{\psi,\varphi}$ is $J_\mu $-normal if and only if$$
	\left\{
	\begin{array}{ll}
		(|c|^2-|a|^2)\bar{\mu}=(|a|^2-|b|^2)\bar{\mu}\\
		(\bar{c}d-\bar{a}b)\bar{\mu}=\bar{a}c-\bar{b}d\\
		c\bar{d}-a\bar{b}=(a\bar{c}-b\bar{d})\bar{\mu}\\
		|d|^2-|b|^2=|d|^2-|c|^2,
	\end{array}
	\right.$$
	which is equivalent to $$|b|=|c| {\rm ~and~} 	(\bar{c}d-\bar{a}b)\bar{\mu}=\bar{a}c-\bar{b}d.$$  The proof is complete.
\end{proof}

  \begin{Theorem}\label{the2}
  	Let $p\in \DD\setminus \{0\}$, $\varphi(z)={az+b\over cz+d }$ be a linear fractional self-map of $\DD$ and $\psi(z)=\beta K_{\sigma(0)}(z)$, where $\beta $  is a non-zero constant and $\sigma(z)={ \bar{a}z-\bar{c} \over -\bar{b}z+\bar{d} }$. Then $W_{\psi,\varphi}$ is $JW_{\xi_p, \tau_p} $-normal if and only if
   $$(|b|^2-|c|^2)p=(-\bar{a}c+\bar{b}d-a\bar{b}+c\bar{d})|p|^2$$   and $$(|a|^2-|d|^2)|p|^2=(\bar{a}c-\bar{b}d)\bar{p}-(\bar{a}b-\bar{c}d)p.$$
  \end{Theorem}

  \begin{proof}
  	Assume first that $W_{\psi,\varphi}$ is $JW_{\xi_p, \tau_p} $-normal. For any $w,z\in \DD$, we have
  		\begin{align*}
  		&	JW_{\xi_p, \tau_p}W_{\psi,\varphi}W_{\psi,\varphi}^\ast k_w(z)=JW_{\xi_p, \tau_p}W_{\psi,\varphi}\overline{\beta K_{\sigma(0)}(w)}K_{\varphi(w)}(z)\\
  			=&	{\beta}JW_{\xi_p, \tau_p}\beta K_{\sigma(0)}(z)K_{\overline{\sigma(0)}}(w)K_{\varphi(w)}(\varphi(z))\\  			=&|\beta|^2K_{\sigma(0)}(w)\overline{\xi_p(\bar{z})}K_{\overline{\sigma(0)}}(\overline{\tau_p{(\bar{z})}})K_{\overline{\varphi(w)}}(\overline{\varphi(\tau_p{(\bar{z})})})\\
  			=&{|\beta|^2\over 1+{cw\over d}}\cdot{\sqrt{1-|p|^2}  \over1-pz }\cdot{ 1 \over 1+{\bar{c}\over \bar{d}}\cdot{ p-\bar{\lambda}z \over 1-pz } }\cdot{ 1 \over 1-{ aw+b\over cw+d}\cdot{ \overline{a\lambda}(\bar{p}-z)+\bar{b}(1-pz) \over \overline{c\lambda}(\bar{p}-z)+\bar{d}(1-pz) } }\\
  			=&{ |\beta|^2|d|^2\sqrt{1-|p|^2} \over (d+cw)[\bar{d}(1-pz)+\bar{c}(p-\bar{\lambda}z)] }\cdot{ 1 \over 1-{ aw+b\over cw+d}\cdot{ \bar{a}p+\bar{b}-\overline{a\lambda}z-\bar{b}pz \over \bar{d}(1-pz)+\bar{c}(p-\bar{\lambda}z) }}\\
  			=&{ |\beta|^2|d|^2\sqrt{1-|p|^2}  \over (A_1w+B_1)z+C_1w+D_1 },
  		\end{align*}

  	where$$
  	\left\{
  	\begin{array}{ll}
  	A_1=(a\bar{b}-c\bar{d})p+(|a|^2-|c|^2)\bar{\lambda}\\
  	B_1=(|b|^2-|d|^2)p+(\bar{a}b-d\bar{c})\bar{\lambda}\\
  		C_1=c\bar{d}-a\bar{b}+(|c|^2-|a|^2)p\\
  		D_1=|d|^2-|b|^2+(d\bar{c}-\bar{a}b)p.
  	\end{array}
  	\right.$$
  	Noting that $\beta K_{\sigma(0)}h=\beta d$ and $$
  	\tau_p{(\overline{\varphi(\sigma(z))})}=\lambda{ p(a\bar{c}-b\bar{d})\bar{z}+p(|d|^2-|c|^2)-(|a|^2-|b|^2)\bar{z}+\bar{a}c-\bar{b}d \over (a\bar{c}-b\bar{d})\bar{z}+|d|^2-|c|^2-\bar{p}(|a|^2-|b|^2)\bar{z}+\bar{p}(\bar{a}c-\bar{b}d ) },$$
  we get
  		\begin{align*}
  			&W_{\psi,\varphi}^\ast W_{\psi,\varphi}JW_{\xi_p, \tau_p}k_w(z)=T_gC_\sigma T_h^\ast T_{\beta K_{\sigma(0)}}^\ast W_{\beta K_{\sigma(0)},\varphi}J\xi_p(z)K_w(\tau_p(z))\\
  			=&\overline{\beta d}T_gC_\sigma W_{\beta K_{\sigma(0)},\varphi}\overline{\xi_p(\bar{z})}K_{\bar{w}}(\overline{\tau_p{(\bar{z})}})\\
  			=&|\beta|^2\bar{d}T_gC_\sigma K_{\sigma(0)}(z)\overline{\xi_p(\overline{\varphi(z)})}K_{\bar{w}}(\overline{\tau_p{(\overline{\varphi(z)})}})\\
  			=&|\beta|^2\bar{d}g(z)K_{\sigma(0)}(\sigma(z))\overline{\xi_p(\overline{\varphi(\sigma(z))})}K_{\bar{w}}(\overline{\tau_p{(\overline{\varphi(\sigma(z))})}})\\
  			=&{ |\beta|^2\bar{d} \over -\bar{b}z+\bar{d} }\cdot{1  \over 1+{c\over d}\cdot{ \bar{a}z-\bar{c} \over -\bar{b}z+\bar{d}}}\cdot{ \sqrt{1-|p|^2} \over 1-p{ (|a|^2-|b|^2)z+b\bar{d}-a\bar{c} \over (\bar{a}c-\bar{b}d)z+|d|^2-|c|^2 } }\cdot{1\over1- w\overline{\tau_p{(\overline{\varphi(\sigma(z))})}}}
  		\end{align*}
  		\begin{align*}
  			=&{|\beta|^2|d|^2 \sqrt{1-|p|^2} \over (\bar{a}c-\bar{b}d)z+|d|^2-|c|^2-p(|a|^2-|b|^2)z -p(b\bar{d}-a\bar{c})}\cdot{1\over1- w\overline{\tau_p{(\overline{\varphi(\sigma(z))})}}}\\
  			=&	{ |\beta|^2|d|^2\sqrt{1-|p|^2}  \over (A_2w+B_2)z+C_2w+D_2 },
  		\end{align*}
  	  	where $$
  	\left\{
  	\begin{array}{ll}
  		A_2=(-\bar{a}c+\bar{b}d)p+(|a|^2-|b|^2)\bar{\lambda}\\
  		B_2=\bar{a}c-\bar{b}d-p(|a|^2-|b|^2)\\
  		C_2=-(|d|^2-|c|^2)p+(b\bar{d}-a\bar{c})\bar{\lambda}\\
  	D_2=|d|^2-|c|^2-p(b\bar{d}-a\bar{c}).
  	\end{array}
  	\right.$$Therefore,
  	$W_{\psi,\varphi}$ is $JW_{\xi_p, \tau_p} $-normal if and only if
  $$
  	(A_1w+B_1)z+C_1w+D_1=(A_2w+B_2)z+C_2w+D_2 $$
  for any $w,z\in \DD $.
   This is equivalent to
    $$A_1=A_2,B_1=B_2,C_1=C_2,D_1=D_2,$$
 which is equivalent to
  $$
  	(|b|^2-|c|^2)p=(-\bar{a}c+\bar{b}d-a\bar{b}+c\bar{d})|p|^2$$
  and
  	$$(|a|^2-|d|^2)|p|^2=(\bar{a}c-\bar{b}d)\bar{p}-(\bar{a}b-\bar{c}d)p.$$
   The proof is complete.
  \end{proof}

  Next we consider the unitary class, Hermitian class and normal class.   Bourdon and Narayan \cite{bn} characterized the symbols $\varphi $ and $\psi$ so that the operator  $W_{\psi,\varphi}$ is unitary on $H^2(\DD)$. In \cite{ck}, Cowen and Ko gave the complete characterization of Hermitian weighted composition operators. Some partial results of the normal weighted composition operators were given by Jiang et al. in \cite{jhz}. Their main results are as follows.

\begin{Lemma}\label{lem4} \cite[Theorem 6]{bn}
	A weighted composition operator $W_{\psi,\varphi}$ is unitary on $H^2(\DD)$ if and only if $\varphi $ is an automorphism of $\DD$ and $\psi=\gamma K_q/\|K_q\|$ where $\varphi(q)=0$ and $|\gamma|=1.$
\end{Lemma}

\begin{Lemma}\label{lem5} \cite[Theorem 2.1]{ck}
	Let $\varphi$ be an analytic self-map of $\DD$ and  $\psi\in H^\infty$. Then the weighted composition operator $W_{\psi,\varphi}$ is Hermitian on $H^2(\DD)$ if and only if $$\varphi(z)=a_0+{a_1z\over 1-\bar{a}_0z} ~{\rm and} ~~\psi(z)={a_2\over 1-\bar{a}_0z}$$ with $ a_0\in \DD$ and $a_1, a_2\in \RR$, where $a_0=\varphi(0)$, $ a_1=\varphi'(0)$ and $a_2=\psi(0).$
\end{Lemma}

\begin{Lemma}\label{lem6} \cite[Lemma 3]{jhz}
		Let $\varphi(z)={az+b\over cz+d }$ be a linear fractional self-map of $\DD$, which has a boundary fixed point, and $\psi(z)= K_{\sigma(0)}(z)$, where $\sigma(z)={ \bar{a}z-\bar{c} \over -\bar{b}z+\bar{d} }$. Then  the weighted composition operator $W_{\psi,\varphi}$ is normal on $H^2(\DD)$ if and only if $|b|=|c|$.
\end{Lemma}

  \begin{Theorem}
    Let $p\in \DD\setminus \{0\}$,
    $$\varphi(z)=\gamma_1{q-z\over 1-\bar{q}z} ~{\rm and}~ \psi(z)=\gamma_2{ \sqrt{1-|q|^2} \over 1-\bar{q}z}$$ with $q\in \DD$ and $|\gamma_1|=|\gamma_2|=1$. Then the following statements hold:
    \begin{enumerate}
    	\item[(i)] $W_{\psi,\varphi}$ is unitary and $J_\mu$-normal;
    	\item[(ii)]$W_{\psi,\varphi}$ is not $JW_{\xi_p, \tau_p} $-normal.
    	\end{enumerate}
  \end{Theorem}

  \begin{proof}
  	Let $$\beta=\gamma_2\sqrt{1-|q|^2}, \, a=-\gamma_1, \, b=\gamma_1q, \, c=-\bar{q}~{\rm and}~d=1.$$
  	
  	$(i)$. Since $ \bar{a}b-\bar{c}d= \bar{a}c-\bar{b}d=0$, then Theorem \ref{the5} and Lemma \ref{lem4} give the desired result.
  	
  	$(ii)$. It is clear that $$(|a|^2-|d|^2)|p|^2=(1-|q|^2)|p|^2\ne0=(\bar{a}c-\bar{b}d)\bar{p}-(\bar{a}b-\bar{c}d)p.$$Thus, $W_{\psi,\varphi}$ is not $JW_{\xi_p, \tau_p} $-normal  by Theorem \ref{the2}.
  \end{proof}

  \begin{Theorem}
  	Let $p\in \DD\setminus \{0\}$, $$\varphi(z)=a_0+{a_1z\over 1-\bar{a}_0z} ~{\rm and} ~~\psi(z)={a_2\over 1-\bar{a}_0z}$$ with $ a_0\in \DD$ and $a_1, a_2\in \RR$, where $\varphi$ maps the unit disk into itself.  Then the following statements hold:
\begin{enumerate}
	\item[(i)] $W_{\psi,\varphi}$   is $J_\mu$-normal and Hermitian if and only if $$ (a_0-\bar{a}_0\mu)(1+a_1+|a_0|^2)=0;$$
\item[(ii)]  $W_{\psi,\varphi}$   is $JW_{\xi_p, \tau_p} $-normal   and Hermitian if and only if $$ [(a_1-|a_0|^2)^2-1]|p|^2=-(\overline{a_1}-|a_0|^2+1)2Re(a_0p).$$
   \end{enumerate}
    \end{Theorem}

  \begin{proof}
  	Let $$\beta=a_2,\, a=a_1-|a_0|^2, \, b=a_0, \, c=-\bar{a}_0 ~{\rm and}~ d=1.$$
 Obviously, we have $|b|=|c|$.

$ (i)$. By a simple calculation, we have
 $$
  	(\bar{c}d-\bar{a}b)\bar{\mu}=(-a_0-a_0{a_1}+a_0|a_0|^2)\bar{\mu}$$and
  	$$\bar{a}c-\bar{b}d=-\overline{a_0}-\overline{a_0}a_1+\overline{a_0}|a_0|^2=\bar{a}c-\bar{b}d.
  	$$
    Theorem \ref{the5} gives that $W_{\psi,\varphi}$   is $J_\mu$-normal if and only if $$(-a_0-a_0{a_1}+a_0|a_0|^2)\bar{\mu}=-\overline{a_0}-\overline{a_0}a_1+\overline{a_0}|a_0|^2,
    $$that is, $ (a_0-\bar{a}_0\mu)(1+a_1+|a_0|^2)=0.$

   $(ii)$. Since
    $$-\bar{a}c+\bar{b}d-a\bar{b}+c\bar{d}=(a_1-|a_0|^2)\bar{a}_0+\bar{a}_0-(a_1-|a_0|^2)\bar{a}_0-\bar{a}_0=0,$$then
    $$(|b|^2-|c|^2)p=(-\bar{a}c+\bar{b}d-a\bar{b}+c\bar{d})|p|^2.$$By a simple calculation, we get
   $$
   	(|a|^2-|d|^2)|p|^2=[(a_1-|a_0|^2)^2-1]|p|^2\\$$and
   	\begin{align*}
   &(\bar{a}c-\bar{b}d)\bar{p}-(\bar{a}b-\bar{c}d)p\\
   	=&[-(\overline{a_1}-|a_0|^2)\overline{a_0}-\overline{a_0}]\bar{p}-[(\overline{a_1}-|a_0|^2){a_0}+{a_0}]{p}\\
  	=&-\overline{a_0p}(\overline{a_1}-|a_0|^2+1)-a_0p(\overline{a_1}-|a_0|^2+1)\\
  		=&-(\overline{a_1}-|a_0|^2+1)(\overline{a_0p}+a_0p)\\
   	=&-(\overline{a_1}-|a_0|^2+1)2Re(a_0p).
   \end{align*}
  Then Theorem \ref{the2} deduces that $W_{\psi,\varphi}$ is $JW_{\xi_p, \tau_p}$-normal if and only if $$ [(a_1-|a_0|^2)^2-1]|p|^2=-(\overline{a_1}-|a_0|^2+1)2Re(a_0p).$$ The proof is complete.
  \end{proof}

  \begin{Theorem}
  	Let $p\in \DD\setminus \{0\}$,  $\varphi(z)={az+b\over cz+d }$ be a linear fractional self-map of $\DD$, which has a boundary fixed point, and $\psi(z)=\beta K_{\sigma(0)}(z)$, where $\beta $  is a constant and $\sigma(z)={ \bar{a}z-\bar{c} \over -\bar{b}z+\bar{d} }$.   If $W_{\psi,\varphi}$ is normal, then the following statements hold:
  	\begin{enumerate}
  		\item[(i)] $W_{\psi,\varphi}$ is $J_\mu$-normal if and only if $(\bar{c}d-\bar{a}b)\bar{\mu}=\bar{a}c-\bar{b}d$,
  		\item[(ii)]$W_{\psi,\varphi}$ is $JW_{\xi_p, \tau_p}$-normal   if and only if $a\bar{b}-c\bar{d}=\bar{b}d-\bar{a}c$ and $$(|a|^2-|d|^2)|p|^2=2Re((a\bar{c}-d\bar{d})p).$$
  	\end{enumerate}
  \end{Theorem}

  \begin{proof}
  	Since $W_{\psi,\varphi}$ is normal, using Lemma \ref{lem6}, we have $|b|=|c|$.
  	
  	$(i)$. Theorem \ref{the5} gives the desired result.
  	
  	$(ii)$. Theorem \ref{the2} gives that $W_{\psi,\varphi}$ is $JW_{\xi_p, \tau_p} $-normal if and only if
  	$$(|b|^2-|c|^2)p=(-\bar{a}c+\bar{b}d-a\bar{b}+c\bar{d})|p|^2$$   and $$(|a|^2-|d|^2)|p|^2=(\bar{a}c-\bar{b}d)\bar{p}-(\bar{a}b-\bar{c}d)p.$$
  	
  	Since $ |b|=|c|$, then $(-\bar{a}c+\bar{b}d-a\bar{b}+c\bar{d})|p|^2=0$, that is $$a\bar{b}-c\bar{d}=\bar{b}d-\bar{a}c.$$
  	Therefore,
  	\begin{align*}
  		(|a|^2-|d|^2)|p|^2=&(\bar{a}c-\bar{b}d)\bar{p}-(\bar{a}b-\bar{c}d)p\\
  		=&(\bar{a}c-\bar{b}d)\bar{p}+(a\bar{c}-d\bar{d})p\\=&2Re((a\bar{c}-d\bar{d})p).
  		\end{align*}
   The converse is obvious by Theorem \ref{the2}.  The proof is complete.
  \end{proof}

\end{document}